\documentclass[11pt]{amsart}

\usepackage{amsmath,amssymb,amsthm}
\usepackage{amsrefs}
\usepackage{enumerate}
\numberwithin{equation}{section}
\newtheorem{thm}[equation]{Theorem}
\newtheorem*{thm*}{Theorem}

\newtheorem{prop}[equation]{Proposition}

\newtheorem{lemma}[equation]{Lemma}
\newtheorem{ex}[equation]{Example}
\newtheorem{coro}[equation]{Corollary}
\newtheorem*{coro*}{Corollary}

\theoremstyle{remark}
\newtheorem*{remark*}{Remark}
\newtheorem{remark}[equation]{Remark}

\theoremstyle{definition}

\numberwithin{figure}{section}
\numberwithin{table}{section}

\begin{document}

\begin{abstract}
For primes $p,e>2$ there are at least $p^{e-3}/e$ groups of order $p^{2e+2}$ 
that have equal multisets of isomorphism types of proper subgroups and proper quotient
groups, isomorphic character tables, and power maps.  This  obstructs recent
speculation concerning a path towards efficient isomorphism tests for general finite groups.
These groups have a special purpose polylogarithmic-time isomorphism test.  
\end{abstract}

\title{The threshold for subgroup profiles to agree is $\Omega(\log n)$.}
%\markright{}
\author{James B. Wilson}
\thanks{This research was supported in part by  NSF grant DMS 1620454.}
\address{
	Department of Mathematics\\
	Colorado State University\\
	Fort Collins, CO 80523\\
}
\email{James.Wilson@ColoState.Edu}
\date{\today}

\maketitle

\section{Introduction}
A recent breakthrough result by Babai has pushed the complexity of 
isomorphism testing of finite graphs on $n$ vertices to an upper bound of
$n^{O((\log n)^c))}$ for some $c\geq 1$ \cite{Babai:graph}. This brings the 
complexity of graph isomorphism within range of the present complexity 
for isomorphism testing of groups of order $n$.  That complexity is bounded by
$n^{O(\mu(n))}$ where given the prime factorization $n=p_1^{e_1}\cdots p_s^{e_s}$,
\begin{align*}
	\mu(p_1^{e_1}\cdots p_s^{e_s}) & =\max\{e_1,\dots, e_s\}.
\end{align*}
Pultr and Hedrl\'in \cite{PH} constructed reductions that imply that group isomorphism reduces to graph isomorphism in time 
polynomial in $n$; see also   \cite{Miller}.  When $\mu(n)$ is bounded then group isomorphism is in polynomial
time in $n$.  For each $c>1$, as $n\to\infty$, the number of integers $n$ for which 
$\mu(n)\leq c$ tends to $1/\zeta(c)$, e.g. 60\% of integers are square-free and 99\% have $\mu(n)\leq 8$.  
Yet for $n=p^{\ell}$ the group isomorphism problem  has the complexity of $n^{O(\log_p n)}$
which makes it an obstacle to the improvement of graph isomorphism. 

In that vein recent speculation by Gowers \cite{Gowers:blog} and Babai \cite{Babai:graph}*{p. 81}
has revisited the idea of using
a portion of the subgroups of a finite group to determine isomorphism types
of finite groups.  Algorithms for testing isomorphism have been successfully 
using such ideas as heuristics for some time; cf. fingerprinting in \cite{ELGOB}.  
Yet, proving efficiency based on these heuristics has been obstructed by knowledge of
examples of Rottl\"ander, and  others, which show that lattices are not enough to 
characterize isomorphism  \cite{Rottlaender}.

Circumventing existing counter-examples, Gowers introduced a threshold criterion.   
In \cite{Gowers:blog} he asked if as the 
number $d(G)=\min\{d : G=\langle x_1,\dots,x_d\rangle\}$ grows toward the 
upper bound of $\log_2 |G|$ (actually $\mu(|G|)+1$ \cite{Guralnick}), 
is the isomorphism type of $G$ determined by the subgroups that 
are $k$-generated, for a $k$ much smaller than $d(|G|)$?  If true it would 
improve isomorphism testing to $O(|G|^{2k})$ steps.  
Glauberman-Grabowski \cite{GG} gave examples $G$ where
$k\geq \sqrt{2\log_3 |G|} - 5/2$.  We will give examples of groups $G$ of
odd order $|G|=p^{\ell}$ for which we need $k=\ell-2$. So in general 
we need $k\geq\log_3 |G|-2$.

Fix primes $p,e>2$. The Heisenberg group over a field $\mathbb{F}_{p^e}$ of
order $p^e$ is:
\begin{align*}
	H & = H(\mathbb{F}_{p^e}) = \left\{\begin{bmatrix} 1 & \alpha & \gamma \\ & 1 & \beta\\ & & 1\end{bmatrix}
		:\alpha,\beta,\gamma\in \mathbb{F}_{p^e}\right\}.
\end{align*}
By the {\em subgroup profile} of a group $G$ we mean the partition of the proper
subgroups of $G$ into isomorphism classes.   Likewise define the 
{\em quotient-group profile}.  We prove:

\begin{thm}\label{thm:main}
The groups $N\leq H'$ of index $p^{2e+2}$ in $H$ are normal in $H$ and 
each $H/N$ has the same subgroup and quotient-group profile and has
$d(H/N)=2e$.  Yet,
\begin{align*}
	\mathcal{G}_{p,e} & = \{H/N : N\leq H',  |H:N|=p^{2e+2}\}
\end{align*} 
has at least $p^{e-3}/e$ isomorphism classes.
\end{thm}

%A group $G$ of order $p^{\ell}$ with $d(G)=\ell$ is isomorphic to $\mathbb{Z}_p^{\ell}$
%and so is distinguished by it subgroup $2$-profile.    
%Groups with $d(G)=\ell-1$ decompose as central products of an extraspecial $p$-group
%and an abelian $p$-group. These too are distinguished by their subgroup $2$-profiles.  
%So the examples of Theorem~\ref{thm:main} are best possible for $p>2$.

\subsection{Further invariants}
More can be said about the similarities in the groups of Theorem~\ref{thm:main}.
Brauer had asked if non-isomorphic groups could have isomorphic character tables
together with exponent structure.
Dade offered the first counter-examples \cite{Dade}.  
The groups in $\mathcal{G}_{p,e}$ also have isomorphic character tables and together with
$p$-th power maps \cite{LW}.  Indeed, these examples have the largest possible 
character tables with that property,  specifically of size $\frac{n}{p^2}\times \frac{n}{p^2}$ (the largest
any character table can be is $n\times n$).
All noncentral conjugacy classes in $H/N$ have the same size.
The groups are both directly and centrally indecomposable and with the same algebraic 
type of indecomposability (an invariant introduced in \citelist{\cite{Wilson:unique-cent}*{Theorem~4.41}\cite{Wilson:RemakI}*{Theorem~8}} 
that links indecomposability to isomorphism types of
local commutative rings and local Jordan algebras). 

Barnes-Wall \cite{Barnes-Wall} show that the lattice of a nilpotent
group of class $2$ and exponent $p$ determines the isomorphism
type of the group
(which corrects an errant remark of the author).
The groups in $\mathcal{G}_{p,e}$ have maximum sized lattices 
with $|G|^{\Theta(\log |G|)}$ subgroups, chains of length 
$\log_p |G|$ and antichains of 
length $|G|^{\Theta(\log_p |G|)}$.
We have no tools to compare such large lattices.

Despite similarities, isomorphism in $\mathcal{G}_{p,e}$ is easy to test.

\begin{thm}[\citelist{\cite{LW}\cite{BMW}}]\label{thm:algo}
\begin{enumerate}[(a)]
\item There is a deterministic algorithm that, given a black-box group $G$,
determines if $G\cong H/N$ for some $N<H$ and if so returns a surjection
$H\to G$.  The timing is polynomial in $e+p$.

\item There is a deterministic algorithm that given groups $G_1,G_2\in
\mathcal{G}_{p,e}$, decides if $G_1\cong G_2$.  The timing is polynomial
in $e+p$.

\end{enumerate}
\end{thm}

We leave discussion of computational models for groups to the references just cited, which we
note involve algorithms that apply to a broader class of problems.  Narrowed to our specific
setting where $p$ and $e$ are prime, the precise complexity is $O(p+e^{\omega}\log_2^2 p)$ where $2\leq \omega<3$ is the exponent of feasible matrix multiplication.  The leading $p$ can be replaced by $\log^2 p$ at the cost
of a Las Vegas polynomial-time algorithm.
It takes $\Omega((e\log_2 p)^2)$ bits to input the groups we consider
by any of the standard methods including matrices, presentations, or permutations.
\medskip

Some of the steps in the proof of Theorem~\ref{thm:main} can be extracted as special cases of results
in \citelist{\cite{LW}\cite{BW:autotopism}\cite{BMW}}.  However, there is an increased need to provide an easier introduction
into the methods represented in those works.  We opt to make this exposition largely self-contained and
we rely in as much as possible on proofs based in linear algebra.

\subsection*{Preliminaries}
We assume all groups in this note are finite.
The {\em Frattini subgroup} $\Phi(G)$ is the intersection of the maximal subgroups of $G$.
The {\em exponent} of a group $G$ is the least positive integer $m$ such that for every $g\in G$,
$g^m=1$.  The commutator subgroup $H'$ is the smallest normal subgroup whose quotient is abelian,
equivalently the subgroup generated by commutators $[x,y]=x^{-1}x^y=x^{-1}y^{-1}xy$,
and $G^p$ is the subgroup generated by $p$-th powers. The {\em genus}
of $G$ is $d(\Phi(G))$.  We require the following:

\begin{thm}[Burnside Basis Theorem]
For a  $p$-group $G$, $\Phi(G)=G' G^p$ and $G/\Phi(G)\cong \mathbb{Z}_p^{d(G)}$.
\end{thm}

\section{A formula for subgroup profiles}
We prove a formula that, under some hypotheses, 
calculates the subgroup profiles in $p$-groups. This allows us to construct groups 
that produce the same profile without need to directly compare the groups.

\begin{thm}\label{thm:count}
Let $G$ be a $p$-group in which ${\rm Aut}(G)$ acts transitively on maximal subgroups and
such that for a maximal subgroup $M$ of $G$, $d(G)=1+d(M)$.  Then for every $J<G$,
the size of $\mathcal{J}(J)=\{ K< G : K\cong J\}$ is
\begin{align*}
	\sum_{f=0}^{d(M)} \frac{p^{1+d(M)}-1}{p^f-1} \left|\left\{ K\leq M : 
		\begin{array}{c} K\cong J,\\~ |M:K\Phi(M)|=p^f\end{array}\right\}\right|.
\end{align*}
In particular the profile map $J\mapsto |\mathcal{J}(J)|$ depends only the isomorphism type of a
maximal subgroup of $G$.
\end{thm}

\begin{lemma}\label{lem:Frattini}
In a $p$-group $G$ with a maximal subgroup $M$ having $d(G)=1+d(M)$, it follows
that $\Phi(G)=\Phi(M)$.
\end{lemma}
\begin{proof}
Using the Burnside Basis Theorem on $G$ and on $M$ we calculate: 
\begin{align*}
	1 & = \frac{|G:\Phi(G)|\cdot |\Phi(G)|}{|G:M|\cdot |M:\Phi(M)|\cdot |\Phi(M)|}
		=\frac{p^{d(G)} |\Phi(G)|}{ p^{1+d(M)} |\Phi(M)|}
		=\frac{|\Phi(G)|}{|\Phi(M)|}.
\end{align*}
As $\Phi(M)=M' M^p\leq G' G^p=\Phi(G)$, we find that $\Phi(M)=\Phi(G)$.
\end{proof}

\begin{proof}[Proof of Theorem~\ref{thm:count}]
Fix $J<G$.  We use an ${\rm Aut}(G)$-invariant partition:
\begin{align*}
	\mathcal{J}(J) & = \bigcup_{f=1}^{d(G)} \mathcal{J}(J,f),
&	\mathcal{J}(J,f) & = 
		\{K \in \mathcal{J}(J) : |G:K\Phi(G)|=p^f\}.
\end{align*}
Let $\mathcal{M}$ be the set of maximal subgroups of $G$.

Fix $f$ and define a bipartite graph between the two sets $\mathcal{J}(J,f)$ and $\mathcal{M}$, 
such that $(K,X)\in \mathcal{J}(J,f)\times \mathcal{M}$ is an edge if, and only if, $K\leq X$.
The action of ${\rm Aut}(G)$ on this graph permutes the vertices of $\mathcal{M}$ transitively.  
In particular, the degree of every vertex $X\in \mathcal{M}$  the same as
the degree of $M$.  Apply Lemma~\ref{lem:Frattini} to conclude
that $\Phi(G)=\Phi(M)$.  Thus, for every $K\leq M$, $K\Phi(G)=K\Phi(M)$ and so
\begin{align*}
	{\rm deg}~M & = |\{ K\leq M : K\cong J, |G:K\Phi(G)|=p^f \} |\\
		& = |\{ K\leq M : K\cong J,\, |M:K\Phi(M)|=p^{f-1} \} |.
\end{align*}

Next we compute the degree of $K\in \mathcal{J}(J,f)$, i.e. the
size of the set:
\begin{align*}
	\{ X\in \mathcal{M} : K\leq X\}
		& = 	\{ X\in \mathcal{M} : K\Phi(G)\leq X\}.
\end{align*}
Since $G/\Phi(G)\cong \mathbb{Z}_p^{d(G)}$ and $|G:K\Phi(G)|=p^f$
it follows that:
\begin{align*}
	 (G/\Phi(G))/(K\Phi(G)/\Phi(G))\cong 
		\mathbb{Z}_p^{f}.
\end{align*}
In particular the number maximal subgroups of $G$ containing $K$ 
equals the number of hyperplanes in an $f$-dimensional $\mathbb{Z}_p$-vector
space.  

At this point we count the number of edges in our graph in 
two ways.
\begin{align*}
	\frac{p^{d(G)}-1}{p-1}{\rm deg}~M & = 
		\sum_{X\in \mathcal{M}} {\rm deg}~X
		 =	\sum_{K\in \mathcal{J}(J,f)} {\rm deg}~K
		 = |\mathcal{J}(J,f)|\frac{p^{f}-1}{p-1}.
\end{align*}
Thus $|\mathcal{J}(J,f)|= \frac{p^{1+d(M)}-1}{p^{f}-1} {\rm deg}~M$.  The claim follows.
\end{proof}

\section{Making $p$-groups with matrices}\label{sec:making}
We are interested in quotients of groups of $(3\times 3)$-matrices, but it will be easier to discuss 
properties of a larger class of groups.
For that we use a general constructions of $p$-groups that has roots in 
studies of Brahana and Baer \citelist{\cite{Brahana}\cite{Baer}}.  Fix 
a set $\{L_1,\dots,L_t\}$ of $(r\times s)$-matrices 
and define the following group of matrices.  Here and throughout empty blocks
in matrices are presumed to be $0$.
\begin{align}\label{eq:Brahana}
	B(L_1,\dots,L_t) & =
	\left\{
	\left[\begin{array}{c|ccc|c}
	1 & & a & & c \\
	\hline
		& & & & \\	 
	  & & I_r & & \begin{array}{ccc} L_1b^t & \cdots & L_t b^t\end{array} \\
	  & & & &  \\
	 \hline
	  & & & & \\
	  & & & & I_t\\
	  & & & & 
	\end{array}\right]
	: \begin{array}{l} 
	a\in \mathbb{Z}_p^r\\ 
	b\in\mathbb{Z}_p^s\\ 
	c\in\mathbb{Z}_p^t
	\end{array}\right\}.
\end{align}
The dimension of $\langle L_1,\dots,L_t\rangle$ is the genus $g$ of the group.  If $r,s,g\approx n/3$ then this construction
already defines $p^{n^3/27+\Theta(n^2)}$ isomorphism types of groups of order $p^n$, which is
approximately a square-root of all the possible groups of order $p^n$.  So despite humble
appearance, this family is extremely complex.  Our most
important examples will be the Heisenberg groups.

As our fields $\mathbb{F}_q$ are finite, there exists an $\omega\in \mathbb{F}_q$ such that 
$\mathbb{F}_q=\mathbb{Z}_p(\omega)$.  In particular, $\{1,\omega,\dots,\omega^{e-1}\}$ is
a basis for $\mathbb{F}_q$ as a $\mathbb{Z}_p$-vector space.  
Define $m(\omega)_{ij}^{(k)}\in\mathbb{Z}_p$ as the constants such that:
\begin{align*}
	\omega^i\cdot \omega^j &= \sum_{k=0}^{e-1} m(\omega)_{ij}^{(k)} \omega^k.
\end{align*}
Also let $M(\omega)^{(k)}\in\mathbb{M}_e(\mathbb{Z}_p)$ be such that 
$[M(\omega)^{(k)}]_{(i+1)(j+1)}=m(\omega)_{ij}^{(k)}$. 
\begin{ex}
If $\mathbb{F}_{p^e}=\mathbb{Z}_p(\omega)$ then 
$H(\mathbb{F}_q)  \cong B(M(\omega)^{(0)}, \dots, M(\omega)^{(e-1)})$.
\end{ex}

In the following section we will use the groups $B(L_1,L_2)$ to give an alternative description of
the groups $H/N\in \mathcal{G}_{p,e}$.  We will prove:

\begin{thm}\label{thm:genus-2}
The groups in $\mathcal{G}_{p,e}$ are isomorphic to the groups $B(L_1,L_2)$ where $\{L_1,L_2\}$ 
 is a linearly independent set of $(e\times e)$-matrices with $L_1$ invertible and $L_1^{-1}L_2$
 has an irreducible minimum polynomial of degree $e$.
\end{thm}

\subsection{Subgroups by row, column, and matrix elimination.}
One way to explore the subgroups of the groups $B(L_1,\dots,L_g)$ is to restrict the range of values
of $a$ or $b$ in the formula given in \eqref{eq:Brahana}.   For instance, suppose
we restrict the coordinate $a_i=0$.  The result is that the values in the $i$-th row of each matrix 
$L_1, \dots, L_g$ can be ignored within that subgroup.  Hence the subgroup we get is isomorphic
to the group we obtain by first removing the $i$-th row of each matrix in $\{L_1,\dots,L_g\}$ and
then using the construction of \eqref{eq:Brahana} to create a group on these smaller matrices.  
Removing one row produces a maximal subgroup,
two rows a subgroup of index $p^2$, and so on.  The similar idea applies to columns.
Reversing the process and inserting rows or columns creates subgroup embeddings.

Restricting values of $c$ may result in a subset that is not closed to multiplication.
An easy way to avoid that concern is to eliminate entries $c_i$ only once 
the corresponding matrix $L_i=0$.  

\begin{ex}
Using row, column, and matrix insertion,
we embed $H(\mathbb{Z}_3)$ into $H(\mathbb{F}_{9})$.  In this example we
let
$\mathbb{F}_9=\mathbb{Z}_3[x]/(x^2+1)$.   We partition the matrices to help
identify the row or column insertions. 
\begin{align*}
	H(\mathbb{Z}_p)  \cong B([1]) & \hookrightarrow B([1],[0]) \\
		& \hookrightarrow B([1 | 0],[0 | 1]) \\
		& \hookrightarrow B\left(\begin{bmatrix} 1 & 0 \\ \hline 0 & -1 \end{bmatrix},
								\begin{bmatrix} 0 & 1 \\ \hline 1 & 0 \end{bmatrix}\right)
								\cong H(\mathbb{F}_{p^2}).
\end{align*}
\end{ex}

We emphasize that this approach is not guaranteed to explore every subgroup, but it
is nevertheless a good place to begin. 
\medskip

Next we can construct a family of groups each having a maximal subgroup
of a fixed isomorphism type.  As in \cite{JacII}*{p. 70}, for a polynomial $a(t)=a_0t^0+\cdots +a_{e-1}t^{e-1}+t^e\in \mathbb{Z}_p[t]$,
the  {\em companion matrix} will be:
\begin{align*}
	C(a(t)) & = 	\begin{bmatrix}
		0 & 1 & \\
		 & \ddots & \ddots\\
		  & & 0 & 1\\
		a_0 & \cdots & & a_{e-1}
	\end{bmatrix}.
\end{align*}

\begin{lemma}\label{lem:flat}
For a polynomial $a(t)$ of degree $e$, the group 
$G=B(I_e,C(a(t)))$ has a maximal subgroup $M$
whose isomorphism type depends only on $p$ and $e$
and $d(G)=1+d(M)$.
\end{lemma}

\begin{proof} 
We delete the last row of $I_e$ and $C(a(t))$ to obtain:
\begin{align*}
 	M=B\left(
	\overbrace{\begin{bmatrix}
		1 & 0  & \\
		 & \ddots & \ddots\\
		 & & 1 & 0
	\end{bmatrix}}^e,
	\overbrace{\begin{bmatrix}
		0 & 1 & \\
		 & \ddots & \ddots\\
		 & & 0 & 1
	\end{bmatrix}}^e
	\right)
	\hookrightarrow \qquad \qquad\\
	\qquad\qquad
 	B\left(\begin{bmatrix}
		1 & 0  & \\
		 & \ddots & \ddots \\
		 & & 1 & 0\\
		 \hline
		 0 &\cdots & & 1
	\end{bmatrix},
	\begin{bmatrix}
		0 & 1 & \\
		 & \ddots & \ddots\\
		  & & 0 & 1\\
		  \hline
		a_0 & \cdots & & a_{e-1}
	\end{bmatrix}
	\right)=G.
\end{align*}
Evidently $|M:\Phi(M)|=p^{2e-1}=p^{d(G)-1}$.
\end{proof}

\subsection{Quotient groups by linear combinations.}\label{sec:sub-elim}
Next we will want to explore some of quotient groups of $B(L_1,\dots, L_g)$.  One can see in 
\eqref{eq:Brahana} that for each subset $\{i_1,\dots,i_s\}\subseteq \{1,\dots,g\}$, there
is a natural surjection $B(L_1,\dots,L_g)\to B(L_{i_1},\dots,L_{i_s})$ and that is indeed a group
homomorphism.  This is an analogue to the way we created subgroups in the previous 
section.

Likewise,
fix scalars $(a_1,\dots,a_g)$.  Then there is a surjective  homomorphism 
$B(L_1,\dots,L_g) \to B(a_1L_1+\cdots +a_g L_g)$.
More generally given a $(g'\times g)$-matrix $A$, there is a surjective homomorphism:
\begin{align*}
	B(L_1,\dots,L_g) & \mapsto B\left(\sum_{j=1}^g A_{1j}L_j,\dots, \sum_{j=1}^g A_{g'j} L_j\right).
\end{align*}

\subsection{Notable isomorphisms.}
There are also direct ways to create groups isomorphic to $B(L_1,\dots,L_g)$. 
For example, for invertible matrices $X\in\mathbb{M}_n(\mathbb{Z}_p)$ and 
$Y\in\mathbb{M}_m(\mathbb{Z}_p)$,
\begin{align*}
	B(L_1,\dots,L_g)& \cong B(XL_1Y^{t},\dots,XL_gY^{t}).
\end{align*}
We may also permute the order of the matrices.  In particular we can always insist
the first matrix have largest rank and that it be expressed in the form 
$\left[\begin{smallmatrix} I_r & 0 \\ 0 & 0 \end{smallmatrix}\right]$ through Gaussian elimination.
Thus the groups $B(L_1)$ can be classified up to isomorphism by the rank of $L_1$.

More generally, for each $i,j\in\{1,\dots,g\}$,
and $s\in \mathbb{Z}_p^{\times}$,
\begin{align*}
	B(L_1,\dots,L_g)& \cong B(L_1,\dots, L_i+s L_j,\dots,L_g) \cong B(L_1,\dots,sL_i,\dots,L_g).
\end{align*}
Thus, if  $\{L'_1,\dots,L'_g\}$ is another basis for $\langle L_1,\dots,L_g\rangle$,
then $B(L_1,\dots,L_g)\cong B(L'_1,\dots,L'_g)$.  There can be further isomorphisms between these
groups, but these will suffice for our present discussion.  
\medskip

Using these observations we can make even more complex embeddings.

\begin{lemma}\label{lem:b(t)-a(t)}
For every $a(t),b(t)\in \mathbb{Z}_p[t]$ with $\deg b(t)=:f$ and less than $e:=\deg a(t)$, 
there is an embedding $B(I_f, C(b(t)))$ into $B(I_e,C(a(t)))$.
\end{lemma}
We emphasize that $b(t)$ has no relation to $a(t)$ other than having lower degree.  So there is
no algebraic reason to guess at the possible embedding of $B(I_f,C(b(t)))$ into $B(I_e,C(a(t)))$.
Yet with matrices it is an easy calculation.
%\begin{proof}

\smallskip
\noindent{\em Proof.}
Fix $b_0,\dots, b_{e-2}\in\mathbb{Z}_p$.
\begin{align}\label{eq:L1Y}
	\begin{bmatrix} 1 & 0 & & \\ & \ddots & \ddots \\ & & 1 & 0 \end{bmatrix}
	\begin{bmatrix} 1 & & \\ & \ddots & \\ & & 1 \\ b_0 & \dots & b_{e-2} & 1\end{bmatrix}
	& = 	\begin{bmatrix} 1 & 0 & & \\ & \ddots & \ddots &\\ & & 1& 0\end{bmatrix},\\
\label{eq:L2Y}
	\begin{bmatrix} 0 & 1 & & \\ & \ddots & \ddots \\ & & 0 & 1 \end{bmatrix}
	\begin{bmatrix} 1 & & \\ & \ddots & \\ & & 1 \\ b_0 & \dots & b_{e-2} & 1\end{bmatrix}
	& = 	\begin{bmatrix} 0 & 1 & & \\ & \ddots & \ddots &\\ b_0 &\dots & b_{e-2} & 1\end{bmatrix}.
\end{align}
Thus, setting $b(t)=b_0t^0 +\cdots +b_{e-2}t^{e-2}+t^{e-1}$, we obtain the following embedding.
 For the isomorphism we are using the identity $B(L_1Y^t,L_2Y^t)\cong B(L_1,L_2)$
following the calculation of \eqref{eq:L1Y} and \eqref{eq:L2Y}.
\begin{align*}
	B(I_{e-1}, C(b(t))) & \hookrightarrow
	 B\left(\left[\begin{array}{ccc|c} 1 & 0 & & 0 \\ & \ddots & \ddots &\vdots \\ & & 1& 0\end{array}\right],
	\left[\begin{array}{ccc|c} 0 & 1 & & 0 \\ & \ddots & \ddots & \vdots\\ b_0 &\dots & b_{e-2} & 1\end{array}\right]\right)\\
	&\cong B\left(\begin{bmatrix} 1 & 0 & & \\ & \ddots & \ddots \\ & & 1 & 0 \end{bmatrix},
	\begin{bmatrix} 0 & 1 & & \\ & \ddots & \ddots \\ & & 0 & 1 \end{bmatrix}\right)\\
	\pushQED{\qed} 
	& \hookrightarrow B(I_e, C(a(t))).\qedhere
	\popQED
\end{align*}
%\end{proof}

In light of Theorem~\ref{thm:genus-2}, Lemma~\ref{lem:b(t)-a(t)} shows that the groups $H/N\in \mathcal{G}_{p,e}$
each have  $p^{\Omega(e)}$ isomorphism types of proper subgroups.

\begin{prop}\label{prop:pairs}
Given $(e\times e)$-matrices $(L_1,\dots,L_g)$ with $L_1$ invertible, there 
are polynomials $a_1(t)|\cdots |a_m(t)$ and matrices $\tilde{L}_3,\dots,\tilde{L}_g$ such that
\begin{align*}
	B(L_1,L_2) & = B(I_e, C(a_1(t))\oplus \cdots \oplus C(a_m(t)), \tilde{L}_3,\dots,\tilde{L}_g).
\end{align*}
\end{prop}
\begin{proof}
Use the Frobenius Normal Form \cite{JacII}*{p. 93} to find a divisor chain $a_1(t)|\cdots |a_m(t)$ 
and an invertible matrix $X$ such that 
\begin{align*}
	X^{-1}(L_1^{-1}L_2)X& = C(a_1(t))\oplus \cdots \oplus C(a_m(t)).
\end{align*}
Hence,
\begin{align*}
	B(L_1,L_2,\dots,L_g) & \cong 
		B(I_e,L_1^{-1}L_2,\dots,L_1^{-1}L_g) \\
		& \cong B(I_e,C(a_1(t))\oplus \cdots \oplus C(a_m(t)),\dots, X^{-1}L_1^{-1}L_g X).
\end{align*}
So for $3\leq i\leq g$, set $\tilde{L}_i=X^{-1} L_1^{-1}L_2 X$.
\end{proof}

%===============================
\section{Isomorphisms between quotients of Heisenberg groups.}
We have so far created many groups and demonstrated the ease to which we can control
the construction of interesting subgroups and quotients.  Our effort now shifts back to 
Heisenberg groups and in particular we will tackle the question of isomorphisms and automorphisms
within $\mathcal{G}_{p,e}$.  Our main results in this section 
are proofs of Theorem~\ref{thm:genus-2} and:

\begin{thm}\label{thm:adj-tensor}
Every isomorphism between nonabelian quotients of $H$ of genus $g>1$ lifts to
an automorphism of $H$.
\end{thm}

This is a special case of \cite{LW}*{Theorem~4.4}.  Here we provide a self-contained and
largely matrix-based proof.

\subsection{The role of commutation.}
The first principle in nilpotent group theory is to
treat groups like rings by invoking commutation $[x,y]=x^{-1}x^y=x^{-1}y^{-1}xy$ as
a skew-commutative multiplication.  This very nearly distributes over the usual product,
in the following way.
\begin{align}\label{eq:comm-id}
	[xy,z] & = [x,z]^y [y,z], & [x,yz] & = [x,z][x,y]^z.
\end{align}
With $q=p^e$ and $H=H(\mathbb{F}_q)$, commutation takes the following form.
\begin{align}\label{eq:comm-H}
	\left[
	\begin{bmatrix}
	1 & \alpha & \gamma\\ & 1 & \beta\\ & & 1 
	\end{bmatrix},
	\begin{bmatrix}
	1 & \alpha' & \gamma'\\ & 1 & \beta'\\ & & 1 
	\end{bmatrix}
	\right] 
	& = 
	\begin{bmatrix}
	1 & 0 & \alpha \beta'-\alpha'\beta\\ & 1 & 0\\ & & 1 
	\end{bmatrix}.
\end{align}
This shows the following two groups are abelian.
\begin{align*}
	H' & = [H,H] =\left\{ \begin{bmatrix} 1 & 0 & \gamma\\ & 1 & 0\\ & & 1 \end{bmatrix} 
	: \gamma \in\mathbb{F}_{q}\right\}, & H/H'\cong \{(\alpha,\beta) : 
	\alpha,\beta\in \mathbb{F}_q^m\}.
\end{align*}
Evidently there are isomorphisms $\iota:H/H'\to \langle \mathbb{F}_q^2,+\rangle$ 
and $\hat{\iota}:H'\to \langle\mathbb{F}_q,+\rangle$.  None of these isomorphisms
is natural in the category of groups.  In particular neither $H/H'$ nor $H'$ is an obvious 
$\mathbb{F}_q$-vector space as scalar multiplication is not part of the operations
of a group.  

Normal subgroups are now easily described.
\begin{lemma}\label{lem:quo}
For $h\in H-H'$, $[h,H]=H'$; thus, if $M$ is normal in $H$ then either
$H'\leq N$ or $N\leq H'$.  In either case, $(H/N)'=H'N/N$.
\end{lemma}
%\begin{proof}
%Let $h=\left[\begin{smallmatrix} 1 & a & c\\ & 1 & b \\ & & 1\end{smallmatrix}\right]$.
%By assumption $(a,b)\neq 0$.  If $a\neq 0$ then for $c\in \mathbb{F}_q$  set
%$k=\left[\begin{smallmatrix} 1 & 0 & 0\\ & 1 & c/a \\ & & 1\end{smallmatrix}\right]$;
%otherwise set
%$k=\left[\begin{smallmatrix} 1 & c/b & 0\\ & 1 & 0 \\ & & 1\end{smallmatrix}\right]$.
%In any case, 
%$[h,k]=\left[\begin{smallmatrix} 1 & 0 & c\\ & 1 & 0 \\ & & 1\end{smallmatrix}\right]$.
%Now if $N$ is normal in $H$ and $h\in N-H'$ then $H'=[h,H]\leq N$. 
%\end{proof}

\subsection{Quotients of $H$}
To inspect the quotients of $H$ we  use a method
to ``linearize'' a nilpotent group which is in some sense the reversal of the
constructions we gave in Section~\ref{sec:making}.  Early versions of this approach
were described by Brahana and 
Baer \citelist{\cite{Baer}\cite{Brahana}}.

Since elements in $H'=[H,H]$ commute
with the whole group, the identities \eqref{eq:comm-id} imply that
commutation factors through $H/H'\times H/H'\to H'$ and thereby
affords a biadditive map 
$[,]_+:\langle \mathbb{F}_q^2,+\rangle\times \langle \mathbb{F}_q^2,+\rangle\to 
		\langle \mathbb{F}_q,+\rangle$:
\begin{align}\label{eq:comm-1}
	[(\alpha,\beta),(\alpha',\beta')]_+ & = \alpha \beta'-\alpha'\beta.
\end{align}
To distinguish between the various roles of $[,]$ we let $[,]$ denote group
commutation and $[,]_+$ the biadditive mapping that commutation produces.

\begin{remark}\label{rem:iota}
The expression in \eqref{eq:comm-1} is obviously $\mathbb{F}_q$-bilinear.
However, the relationship of $[,]_+$ to the commutation map
 $[,]:H/H'\times H/H'\to H'$ is only as abelian groups, made explicitly through
the (unnatural) choice of $(\iota,\hat{\iota})$ above.   So geometric information about 
$\mathbb{F}_q$-bilinear maps cannot be directly applied in our situation.
\end{remark}

Now Lemma~\ref{lem:quo} shows that for $N<H'$, $(H/N)'=H'/N$.  So
the commutation of the quotient $H/N$ will accordingly afford a new biadditive map
\begin{align*}
	[,]_+^{H/N}:\langle \mathbb{F}_q^2,+\rangle\times \langle \mathbb{F}_q^2,+\rangle\to 
		\langle \mathbb{F}_q,+\rangle
	\overset{\pi}{\to} \mathbb{Z}_p^g 
\end{align*}
where $\pi$ is given as the homomorphism 
$\langle \mathbb{F}_q,+\rangle\cong H'\to H'/N\cong \mathbb{Z}_p^g$. The genus of $H/N$ is 
the value $g$.  

Let us look closely at the case of genus $g=1$.  Fix $\pi:\langle\mathbb{F}_q,+\rangle\to \mathbb{Z}_p$.
Choose a basis $\{\alpha_1,\dots,\alpha_e\}$ for $\langle \mathbb{F}_q,+\rangle$ as a $\mathbb{Z}_p$-vector
space and such that $\pi(\alpha_i)=1$ if $i=1$ and $0$ otherwise.  Define 
\begin{align*}
	L_{ij}=\pi([(\alpha_i,0),(0,\alpha_j)]) = \pi(\alpha_i \alpha_j).
\end{align*}
Regarded as a map of $\mathbb{Z}_p$-vector spaces we see:
\begin{align*}
	[(\alpha,\beta),(\alpha',\beta')]_+^{H/N} & = 
		\begin{bmatrix} \alpha & \beta\end{bmatrix}
		\begin{bmatrix} 0 & L\\ -L^t & 0 \end{bmatrix}
		\begin{bmatrix} \alpha'\\ \beta'\end{bmatrix}.
\end{align*}

As we vary $H/N$ amongst groups of arbitrary genus $1\leq g\leq e$ 
we describe $[,]_+=[,]_+^{H/N}$ by a linearly independent set of invertible matrices 
$L_1,\dots,L_g\in \mathbb{M}_e(\mathbb{Z}_p)$ such that
\begin{align*}
	[(\alpha,\beta),(\alpha',\beta')]_+ & = 
		\left(
		\begin{bmatrix} \alpha &\beta \end{bmatrix}\!
		\begin{bmatrix} 0 & L_1\\ -L_1^t & 0 \end{bmatrix}\!
		\begin{bmatrix}\alpha'\\ \beta'\end{bmatrix},
		\dots,
		\begin{bmatrix} \alpha &\beta \end{bmatrix}\!
		\begin{bmatrix} 0 & L_g\\ -L_g^t & 0 \end{bmatrix}\!
		\begin{bmatrix}\alpha'\\  \beta'\end{bmatrix}\right).
\end{align*}
This demonstrates the following correspondence.

\begin{thm}[Brahana correspondence]\label{thm:Brahana}
A group $H/N$ whose commutation is described by matrices $(L_1,\dots,L_g)$
has an isomorphism to the group $B(L_1,\dots,L_g)$.
In particular all quotients $H/N$ of genus $1$ are isomorphic.
\end{thm}
\begin{proof}
If $g=1$ we can assume $L_1=I_e$ and so $H/N\cong B(I_e)$.  
The required isomorphism $B(L_1,\dots,L_g)\to H/N$ is as follows.
\begin{align}\label{eq:iso-1}
\left[\begin{array}{c|ccc|c}
	1 & & a & & c \\
	\hline
		& & & & \\	 
	  & & I_e & & \begin{array}{ccc} L_1b^t & \cdots & L_g b^t\end{array} \\
	  & & & &  \\
	 \hline
	  & & & & \\
	  & & & & I_g\\
	  & & & & 
	\end{array}\right]
	\mapsto
	\begin{bmatrix}
		1 & \iota^{-1}(a) & \hat{\iota}^{-1}(c) \\ & 1 & \iota^{-1}(b)\\ & & 1
	\end{bmatrix}
	\mod{N}.		
\end{align}
\end{proof}

\begin{coro}\label{coro:quo-profile}
The groups in $\mathcal{G}_{p,e}$ have the equal quotient group profiles.
\end{coro}
\begin{proof}
Fix $N\leq H'$ with $|H:N|=p^{2e+2}$.  
As $p^{2e+2}=|H:H'|\cdot |H':N|=p^{2e}|H':N|$ we find $|H':N|=p^2$, and
so $H/N$ has genus $2$.  As in Lemma~\ref{lem:quo}, if $N< K\leq H$ 
and $K/N$ is normal in $H/N$, then $K<H'$ or $H'\leq K$.  If 
$H'\leq K$ then $(H/K)/(K/N)\cong H/K\cong \mathbb{Z}_p^f$ where $p^f=|H:K|$.  
This does not depend on the choice of $N$.  The number of choices for $K$
is the number of subgroups in $\mathbb{Z}_p^{2e}$
of index $f$, which again does not depend on $N$.
Otherwise $N<K<H'$ and so $|H':K|=p$.  Thus $(H/N)/(K/N)\cong H/K$ has
genus $1$.  So by Theorem~\ref{thm:Brahana} its isomorphism type is 
fixed and independent of $N$.  Finally, $H'/N\cong \mathbb{Z}_p^2$ so there are 
exactly $p+1$ choices of $K$ with $N<K<H'$.  This is independent of $N$.
\end{proof}

\subsection{Distributive products.}
To prove Theorem~\ref{thm:adj-tensor} we need a brief detour to discuss distributive products.
Take $A\subset \mathbb{M}_{r}(\mathbb{Z}_p)\times \mathbb{M}_{s}(\mathbb{Z}_p)$.
It follows that $\mathbb{M}_{r\times s}(\mathbb{Z}_p)$ decomposes into subspaces as
follows.
\begin{align*}
	\left\langle 
	F^t X-XF^* :\begin{array}{l} 
		X\in\mathbb{M}_{r\times s}(\mathbb{Z}_p)\\
		(F,F^*)\in A
		\end{array}
	\right\rangle
	\oplus
	\{X:\forall (F,F^*)\in A, F^tX=XF^*\}.
\end{align*}
We write $\mathbb{Z}_p^r\otimes_A\mathbb{Z}_p^s$ for the right-hand subspace. 
The projection $\pi_A$ from $\mathbb{M}_{r\times s}(\mathbb{Z}_p)$ onto
$\mathbb{Z}_p^r\otimes_A \mathbb{Z}_p^s$ allows us to define a distributive
{\em tensor} product
\begin{align*}
\otimes & =\otimes_A:\mathbb{Z}_p^r\times \mathbb{Z}_p^s\to \mathbb{Z}_p^r\otimes_A\mathbb{Z}_p^s
&
	u\otimes v & = \pi_A(u^t v).
\end{align*}
Notice for $(F,F^*)\in A$, $\pi_A(F^t X)=\pi_A(XF^*)$ and so we find:
\begin{align*}
	uF\otimes v & = \pi_A(F^t u^tv)=\pi_A(u^t vF^*)=u\otimes (vF^*).
\end{align*}
Consider an example with 
\begin{align}\label{eq:tensor-M2}
	A& =\left\{\left(\begin{bmatrix} \alpha & \beta\\ \gamma & \delta \end{bmatrix},
	\begin{bmatrix} \delta & -\beta\\ -\gamma & \alpha \end{bmatrix}\right):
	\alpha,\beta,\gamma,\delta\in \mathbb{F}_q\right\}.
\end{align}
For $(\alpha,\beta),(\gamma,\delta)\in \langle \mathbb{F}_q^2,+\rangle$,
\begin{align*}
	(\alpha,\beta)\otimes (\gamma,\delta) 
		& = (1,0)\begin{bmatrix}\alpha & \beta\\ 0 & 0\end{bmatrix}\otimes
			(1,0)\begin{bmatrix} \gamma & \delta \\ 0 & 0 \end{bmatrix}\\
		& = (1,0)\otimes \left((1,0)\begin{bmatrix} \gamma & \delta \\ 0 & 0 \end{bmatrix}
		\begin{bmatrix} 0 & -\beta \\ 0 & \alpha \end{bmatrix}\right)
		& = (1,0)\otimes (\alpha\delta-\beta\gamma,0).
\end{align*}
Therefore $(\alpha,\beta)\otimes(\gamma,\delta)\mapsto \alpha\delta-\beta\gamma$
defines an isomorphism
\begin{align*}
	\langle \mathbb{F}_q^2,+\rangle\otimes_A \langle \mathbb{F}_q^2,+\rangle
	& \cong \langle \mathbb{F}_q,+\rangle
\end{align*}
and furthermore $\otimes_A:\langle \mathbb{F}_q^2,+\rangle\times 
\langle \mathbb{F}_q^2,+\rangle \to \langle \mathbb{F}_q^2,+\rangle\otimes_A \langle
\mathbb{F}_q^2,+\rangle$ is equivalent to $[,]_+^H$.  That the commutation
of the Heisenberg group is a tensor product over a matrix ring is at the
core of how Theorem~\ref{thm:main} is possible.
\medskip

In general for a distributive product
$*:\mathbb{Z}_p^r\times \mathbb{Z}_p^s\to \mathbb{Z}_p^t$
a pair $(F,F^*)\in \mathbb{M}_r(\mathbb{Z}_p)\times\mathbb{M}_s(\mathbb{Z}_p)$
is an {\em adjoint} if it satisfies, for all $u\in \mathbb{Z}_p^r$ and $v\in \mathbb{Z}_p^s$, 
$(uF)\circ v=u\circ (vF^*)$.
(This is the same notion of adjoints we find in texts on linear algebra,
cf. \cite{JacII}*{p. 143}, 
but we use it on arbitrary products not just inner products.)
The adjoint identity is linear and so it defines a subspace:
\begin{align*}
	{\rm Adj}(\circ) & = \{(F,F^*) : \forall u\in\mathbb{Z}_p^r, 
	\forall v\in\mathbb{Z}_p^s,\, (uF)\circ v = u\circ(vF^*)\}.
\end{align*}
Under the product $(F,F^*)(G,G^*)=(FG,G^*F^*)$ this makes ${\rm Adj}(\circ)$ 
into a ring.  In fact ${\rm Adj}(\circ)$ is the largest ring  $A$ over 
which the product $\circ$ factors through the tensor $\otimes_A$, more 
precisely:

\begin{thm}[Adjoint-tensor Galois correspondence \cite{BW:autotopism}*{Theorem~2.11}]\label{thm:Galois}
Fix a distributive product $\circ:\mathbb{Z}_p^r\times\mathbb{Z}_p^s\to \mathbb{Z}_p^t$
and $A\subset \mathbb{M}_r(\mathbb{Z}_p)\times \mathbb{M}_s(\mathbb{Z}_p)$.
Then $A\subset {\rm Adj}(\circ)$ if, and only if, there is a homomorphism 
$\hat{\circ}:\mathbb{Z}_p^r\otimes_A\mathbb{Z}_p^s\to \mathbb{Z}_p^t$ such that
$u\circ v=\hat{\circ}(u\otimes v)$.
\end{thm}

Now we refocus on the goal of Theorem~\ref{thm:adj-tensor}. 
\begin{lemma}\label{lem:Jac}
If $\mathbb{F}_{p^e}\subseteq A\subseteq \mathbb{M}_{e}(\mathbb{Z}_p)$ 
and $e$ prime, then $A=\mathbb{F}_{p^e}$ or $\mathbb{M}_e(\mathbb{Z}_p)$.
\end{lemma}
\begin{proof}
Let $V=\mathbb{Z}_p^{e}$ be an $A$-module.  As $\mathbb{F}_{p^e}$ is contained
in $A$, $V$ is also an $\mathbb{F}_{p^e}$-vector space, and it is $1$-dimensional.
Thus, as an $A$-module $V$ is simple.  Now $A$ is also faithfully represented on $V$.
Thus by Jacobson's Density Theorem \cite{JacII}*{p. 262}, $A$ is  ${\rm End}_{Z}(V)\cong M_f(Z)$, 
where $Z\cong\mathbb{F}_{p^s}$ is the center of $A$.  Furthermore, $e=fs$.  As $e$ is
prime either $f=1$ and $A=\mathbb{F}_{p^e}$, or else $f=e$ and $s=1$ which makes 
$A=\mathbb{M}_e(\mathbb{Z}_p)$.  
\end{proof}

\begin{lemma}\label{lem:adj}
If $H/N$ has genus $g>1$ then ${\rm Adj}([,]_+^H)={\rm Adj}([,]_+^{H/N})\cong \mathbb{M}_2(\mathbb{F}_q)$.
\end{lemma}
\begin{proof}
We start by observing some necessary adjoints.  The adjoint-tensor Galois correspondence shows 
${\rm Adj}([,]_+^H)\subset {\rm Adj}([,]_+^{H/N})$.  In our example
we found $\mathbb{M}_2(\mathbb{F}_q)\cong {\rm Adj}([,]_+^H)$.  

Next we know that the
commutation in $H/N$ is given by a set $\{L_1,\dots,L_g\}$ of linearly independent
invertible matrices.
So the linear equations to solve to describe ${\rm Adj}([,]_+^{H/N})$ are the following.
For each $1\leq i\leq g$,
\begin{align*}
	\begin{bmatrix}
		F_{11} & F_{12}\\ 
		F_{21} & F_{22}
	\end{bmatrix}
	\begin{bmatrix}
	0 & L_i\\
	-L_i^t & 0 
	\end{bmatrix}
	& =
	\begin{bmatrix}
	0 & L_i\\
	-L_i^t & 0 
	\end{bmatrix}
	\begin{bmatrix}
		F^*_{11} & F^*_{12}\\ 
		F^*_{21} & F^*_{22}
	\end{bmatrix}^t
\end{align*}
For $i=1$ we get $F_{11}^*=L_1 F_{22}^tL_1^{-t}$, $F^*_{12} =-L_1 F_{12}^t L_1^{-t}$,
$F^*_{21}=-L_1F_{21}^tL_1^{-t}$, and $F^*_{22}=L_1F_{11}^tL_1^{-t}$.   Now $L_2$ adds the further constraint that
$F_{ij}L_1^{-1} L_2=L_1^{-1}L_2F_{ij}$.  

Now consider the algebra
\begin{align*}
	A & = \{ F\in \mathbb{M}_e(\mathbb{Z}_p) : FL_1^{-1} L_2=L_1^{-1} L_2F \}.
\end{align*}
By the previous inclusion we know that $\mathbb{F}_q\subseteq A\subseteq \mathbb{M}_e(\mathbb{Z}_p)$.
If $A=\mathbb{M}_e(\mathbb{Z}_p)$ then $L_2$ commutes with every matrix and thus $L_2$ is a scalar
matrix.  However, $L_2$ and $L_1=I_e$ are linearly independent.  So $L_2$ cannot be scalar.
As a result $A\neq \mathbb{M}_e(\mathbb{Z}_p)$.
By Lemma~\ref{lem:Jac}, $A=\mathbb{F}_q$.  That is, 
\begin{align*}
	{\rm Adj}([,]_+^{H/N}) & \subseteq \left\{
		\left(\begin{bmatrix} \alpha & \beta \\ \gamma & \delta\end{bmatrix},
		\begin{bmatrix} \delta & -\beta\\ -\gamma & \alpha\end{bmatrix}\right):
		\alpha,\beta,\gamma,\delta\in\mathbb{F}_q\right\}.
\end{align*}
So indeed ${\rm Adj}([,]_+^H)={\rm Adj}([,]_+^{H/N})$.
\end{proof}

\begin{proof}[Proof Theorem~\ref{thm:genus-2}]
Fix a group $B(L_1,L_2)$.  If $B(L_1,L_2)$ is a quotient of $H$ then so is $B(L_1)$.
By Corollary~\ref{coro:quo-profile}, $B(L_1)\cong B(I_e)$.  Therefore we may assume $L_1=I_e$
and let $a(t)$ be the minimum polynomial of $L_2$.  As $\{L_1,L_2\}$ are linearly independent
we know that $L_2$ cannot be a scalar matrix and so $a(t)$ has degree at least $2$.

Now let $C(L_2)=\{F \in\mathbb{M}_e(\mathbb{Z}_p) : FL_2=L_2F\}$.
Following the calculation of the adjoint ring above we know that 
\begin{align*}
	{\rm Adj}([,]_+^{H/N}) & = \left\{
	\left(\begin{bmatrix} F_{11} & F_{12}\\ F_{21} & F_{22} \end{bmatrix},
	\begin{bmatrix} F_{22}^t & -F_{12}^t\\ -F_{21}^t & F_{22}^t \end{bmatrix}\right) :
	F_{ij} L_2 = L_2 F_{ij}\right\}\\
	& \cong \mathbb{M}_2(C(L_2)).
\end{align*}
Thus, if $B(L_1,L_2)$ is a quotient of $H$ then $C(L_2)\cong \mathbb{F}_q$.
Since $\mathbb{Z}_p[L_2]\cong \mathbb{Z}_p[t]/(a(t))\subset C(L_2)$ it follows
that $\mathbb{Z}_p[L_2]$ is a subfield of $\mathbb{F}_q$.  As $a(t)$ has degree
greater than $1$ and $\mathbb{F}_q$ has no intermediate fields, it follows that
$\mathbb{Z}_p[L_2]=\mathbb{F}_q$.  Thus $a(t)$ is an irreducible polynomial of degree $e$.

Conversely if $L_2$ is conjugate to $C(a(t))$ then 
${\rm Adj}([,]_+^{H/N})\cong \mathbb{M}_2(\mathbb{F}_q)\cong {\rm Adj}([,]_+^{H})$.  
By Adjoint-Tensor Galois correspondence, the commutation in $B(I_e,L_2)$ factors through 
the tensor product over ${\rm Adj}([,]_+^H)$ which is the commutation of $H$.
Therefore $B(I_e,L_2)$ is a quotient of $H$.
\end{proof}

\subsection{Automorphisms of Heisenberg groups}
Now we need to consider the automorphisms of $H$, assuming $p>2$.
Each automorphism is described by three constituents:
\begin{enumerate}
\item a homomorphism $\tau:\langle \mathbb{F}_q^2,+\rangle 
\to \langle \mathbb{F}_q,+\rangle$,
\item an invertible matrix $\begin{bmatrix}\alpha & \beta\\ \gamma & \delta
\end{bmatrix}$ over $\mathbb{F}_q$, and
\item a field automorphism $\alpha\mapsto \bar{\alpha}$ of $\mathbb{F}_q$.
\end{enumerate}
The corresponding automorphism is as follows.
\begin{align}\label{eq:Aut}
	\begin{bmatrix}
		1 & \alpha' & \gamma'\\ & 1 & \beta' \\ & & 1 
	\end{bmatrix}
	\mapsto
	\begin{bmatrix}
		1 & \overline{\alpha'\alpha+\beta'\gamma} & \overline{(\alpha\delta-\beta\gamma)\gamma'+\tau(\alpha',\beta')}\\ 
		& 1 & \overline{\alpha'\beta+\beta'\delta} \\ 
		& & 1 
	\end{bmatrix}.	
\end{align}
\begin{remark}
Classic knowledge of automorphisms of Heisenberg groups over $\mathbb{Z}_p$, $\mathbb{Z}$, 
$\mathbb{R}$ and $\mathbb{C}$ is largely inapplicable here.  For $\mathbb{R}$ and $\mathbb{C}$ 
the automorphisms are presumed to be smooth, ours have no such restrictions.
As we cautioned in Remark~\ref{rem:iota}, in the case of $\mathbb{F}_q$, 
$[,]_+:\langle \mathbb{F}_q^2,+\rangle\times \langle \mathbb{F}_q^2,+\rangle
\to \langle \mathbb{F}_q,+\rangle$  is $\mathbb{F}_q$-bilinear but 
$[,]:H/H'\times H/H'\to H'$ is only biadditive.  So only the cases
of $\mathbb{Z}$ and $\mathbb{Z}_p$ are immediate by standard geometric methods.
\end{remark}

 We have just seen that 
the commutation of Heisenberg groups is actually a special type of
distributive product, a tensor product.  This means instead of acting
on a biadditive map we can act on a ring
${\rm Adj}([,]_+)\cong \mathbb{M}_2(\mathbb{F}_q)$.

\begin{thm}[Skolem-Noether \cite{JacII}*{p. 237}]
The ring automorphisms of $\mathbb{M}_2(\mathbb{F}_q)$ are 
$X\mapsto T^{-1}\bar{X}T$ where $T$ is an invertible $2\times 2$ matrix
and $\alpha \to \bar{\alpha}$ is a field automorphism of $\mathbb{F}_q$ applied
to each entry of $X$.
\end{thm}
\begin{proof}
First the automorphism $\phi$ will send $\alpha I_2\mapsto \bar{\alpha}I_2$
which gives us the field automorphism $\sigma$.  Replacing $\phi$ with $\phi(X^{\sigma^{-1}})$ we now have an $\mathbb{F}_q$-linear automorphism.
Therefore it maps the minimal right ideal
$\left\{\left[\begin{smallmatrix}\alpha & \beta \\ 0 & 0\end{smallmatrix}\right] : \alpha,\beta\in \mathbb{F}_q\right\}$ 
to another minimal right ideal
$\left\{\left[\begin{smallmatrix} 0 & 0 \\  \gamma & \delta \end{smallmatrix}\right]: \gamma,\delta\in \mathbb{F}_q\right\}$, 
or for some $\nu\in \mathbb{F}_q$, 
$\left\{\left[\begin{smallmatrix}\gamma & \delta \\ \nu \gamma & \nu\delta\end{smallmatrix}\right]:\gamma,\delta\in \mathbb{F}_q\right\}$.
Each of these is a $2$-dimensional
vector spaces over $\mathbb{F}_q$ so that transformation can be given by an
invertible square matrix $T$.  
\end{proof}

\begin{lemma}\label{lem:aut-H}
There is an epimorphism ${\rm Aut}(H)\to {\rm Aut}(\mathbb{M}_2(\mathbb{F}_q))$.  The kernel consists of those automorphisms
that are the identity on $H/H'$.
\end{lemma}
\begin{proof}
Let $\phi:H\to H$ be an automorphism.  Since $\phi([h,k])=[\phi(h),\phi(k)]$,
$\phi$ factors through 
$\mathbb{Z}_p^{2e}\cong H/H'\to H/H'\cong\mathbb{Z}_p^{2e}$.
So we let $T$ be the matrix representing that transformation.  Also
we let $\hat{T}$ be the matrix describing the restriction of $\phi$ to 
$\mathbb{Z}_p^e\cong H'\to H'\cong\mathbb{Z}_p^e$.  Notice $(T,\hat{T})$ 
satisfy
\begin{align*}
	[(\alpha,\beta)T,(\alpha',\beta')T]_+ & = [(\alpha,\beta),(\alpha',\beta')]_+\hat{T}.
\end{align*}
Now take $(F,F^*)\in {\rm Adj}([,]_+)$.  It follows that
\begin{align*}
	[(\alpha,\beta)T^{-1}FT,(\alpha',\beta')]_+
	%& = [(\alpha,\beta)T^{-1}FT,(\alpha',\beta')T^{-1}T]_+
		&= [(\alpha,\beta)T^{-1}F,(\alpha',\beta')T^{-1}]_+\hat{T}\\
%		&= [(\alpha,\beta)T^{-1},(\alpha',\beta')T^{-1}F^*]_+\hat{T}
		& =[(\alpha,\beta),(\alpha',\beta')T^{-1}F^*T]_+.
\end{align*}
In this way ${\rm Aut}(H)$ acts on 
${\rm Adj}([,]_+)\cong \mathbb{M}_{2}(\mathbb{F}_q)$.
We saw that commutation in ${\rm Aut}(H)$ is the same
as the tensor product with ${\rm Adj}([,]_+)$, so every automorphism of
${\rm Adj}([,]_+)$ determines an automorphism of $H$. 
\end{proof}

\begin{proof}[Proof of Theorem~\ref{thm:adj-tensor}]
In the Brahana correspondence we saw that every nonabelian quotient
$H/N$ is determined up to isomorphism  by the matrices $(L_1,\dots,L_g)$
which also define $[,]_+^{H/N}$.   Fix an
isomorphism $\phi:H/N_1\to H/N_2$.  Since
$(H/N_i)/(H/N_i)'\cong H/H'\cong\langle\mathbb{F}_q^2,+\rangle$,
we see $\phi$ determines a matrix $T=\begin{bmatrix} A & B \\ C & D\end{bmatrix}
\in \mathbb{M}_{2e}(\mathbb{Z}_p)$.  Using the fact that
$\langle \mathbb{F}_q,+\rangle\cong 
\langle \mathbb{F}_q^2,+\rangle\otimes_{\mathbb{M}_2(\mathbb{F}_q)}
\langle \mathbb{F}_q^2,+\rangle$, we
define $\Gamma:H\to H$ as follows.
\begin{align*}
	\begin{bmatrix}
		1 & a & \sum_i (a_i,b_i)\otimes (x_i,y_i) \\ & 1 & b \\ & & 1 
	\end{bmatrix}
	& \mapsto
	\begin{bmatrix}
		1 & aA+bC & \sum_i (a_i,b_i)T\otimes (x_i,y_i)T \\ & 1 & aB+bD \\ & & 1 
	\end{bmatrix}.
\end{align*}
From our proof of Lemma~\ref{lem:aut-H} we notice $\Gamma$ is an automorphism of $H$ 
if, and only if, $T^{-1}{\rm Adj}([,]_+)T={\rm Adj}([,]_+)$. 
Since $\phi$ is an isomorphism $H/N_1\to H/N_2$ we know that
\begin{align*}
	T^{-1}{\rm Adj}\left([,]_+^{H/N_1}\right)T={\rm Adj}\left([,]^{H/N_2}_+\right).
\end{align*}
By Lemma~\ref{lem:adj} we know ${\rm Adj}\left([,]_+\right)
={\rm Adj}\left([,]_+^{H/N_i}\right)$.  
\end{proof}

\begin{coro}\label{coro:count}
The set $\mathcal{G}_{p,e}$ has at least $p^{e-3}/e$ isomorphism types.
\end{coro}
\begin{proof}
The number of subgroups $N<H'$ of index $p^2$ is
the number of subspaces of codimension 2 in a vector space $H'\cong \mathbb{Z}_p^e$.
That number is $\frac{(p^e-1)(p^e-p)}{(p^2-1)(p-1)}$.  Meanwhile the
action by ${\rm Aut}(H)$ on $H'$ has size $e(p^e-1)$; see \eqref{eq:Aut}.  So the number
of orbits is at least $p^{e-3}/e$.
\end{proof}

\section{Proof of Theorem~\ref{thm:main}}

\begin{lemma}\label{lem:max-conj}
For every $N\leq H'$, ${\rm Aut}(H/N)$ acts 
transitively on the maximal subgroups of $H/N$.
\end{lemma}
\begin{proof}
The group ${\rm SL}(2,\mathbb{F}_{p^e})$ acts transitively on hyperplanes of $\mathbb{Z}_p^{2e}$
and those coincide with the maximal subgroups of $H$.
Following \eqref{eq:Aut}, this action lifts to ${\rm Aut}(H)$ and is furthermore
the identity on $H'$.  Thus for $N<H'$, this action transfers to $H/N$.  Lastly,
observe that $\Phi(H/N)=H'/N=\Phi(H)/N$, so the maximal subgroups of $H/N$ are the groups
$X/N$ where $X$ is maximal in $H$.
\end{proof}

\begin{proof}[Proof of Theorem~\ref{thm:main}]
Corollary~\ref{coro:count} \& \ref{coro:quo-profile} establish that the set
$\mathcal{G}_{p,e}$  has at  least $p^{e-3}/e$ isomorphism types and that
these groups all have the same quotient group profile.  Finally, 
Lemmas~\ref{lem:max-conj} \& \ref{lem:flat} allow us to invoke 
Theorem~\ref{thm:count} to conclude the proof.
\end{proof}

%\section{Closing remarks}
%We have made several choices in our examples for the sake of simplifying the exposition.  For
%instance, there is no need to insist the $e$ prime.  We need only assume that we look at those
%quotients $H/N\cong B(I_e,C(a(t))$ for which $a(t)$ is irreducible.

\subsection*{Acknowledgment.}
Thanks to Laci Babai \& Gene Luks who asked me about profiles.
Thanks to Bill Kantor for extensive feedback on this article.

\end{document}